\documentclass[11pt]{amsart}

\usepackage{setspace} 
\usepackage{amssymb}
\usepackage[dvips]{color}
\usepackage{amsmath}
\usepackage{amsthm}
\usepackage[a4paper]{geometry}
\usepackage{amsfonts}
\usepackage[all]{xypic} 
\usepackage{bbm}
\usepackage{xcolor}

\usepackage{color}
\usepackage[colorlinks=true,linkcolor=blue,citecolor=blue,pdfpagelabels=false]{hyperref}
\usepackage{cleveref}
\usepackage{url}

\usepackage{hyperref}

\usepackage[colorinlistoftodos]{todonotes}

\allowdisplaybreaks[4] 
\newtheoremstyle{myremark}     {10pt}{10pt}{}{}{\bfseries}{.}{.5em}{}
\newtheoremstyle{myremark}     {10pt}{10pt}{}{}{\bfseries}{.}{.5em}{}

 \newtheorem{thm}{Theorem}[section]
 
 \newtheorem{lem}[thm]{Lemma}
 \newtheorem{prop}[thm]{Proposition}
 \theoremstyle{definition}
 \newtheorem{defn}[thm]{Definition}

  \theoremstyle{myremark}
  \newtheorem{rem}[thm]{Remark}

 \newcommand{\h}{\mathcal{H}}

  \newcommand{\D}{\mathcal{D}}
 
 \newcommand{\J}{\mathcal{J}}
 \newcommand{\I}{\mathcal{I}}
 \newcommand{\M}{\mathcal{M}}

 \newcommand{\F}{\mathcal{F}}

 \newcommand{\B}{\mathbf{B}}

 \newcommand{\Z}{\mathbb{Z}}
 \newcommand{\N}{\mathbb{N}}

 \newcommand{\R}{\mathbb{R}}
 \newcommand{\C}{\mathbb{C}}

 \newcommand{\EssD}{\mathcal{D}}

 \newcommand{\abs}[1]{\left\vert#1\right\vert}
 \newcommand{\set}[1]{\left\{#1\right\}}
 
 \newcommand{\norm}[1]{\left\Vert#1\right\Vert}
 
 \newcommand\inner[2]{\left\langle #1, #2 \right\rangle}

\title[ Bicentralizer of  Mixed $q$-deformed Araki-Woods Algebras]{ Connes' Bicentralizer Problem for Mixed $q$-deformed Araki-Woods Algebras}

\author[Bikram]{Panchugopal Bikram}

\address{School of Mathematical Sciences, National Institute of Science Education and Research, Bhubaneswar, A CI of HBNI, Jatani- 752050, India}
\email{bikram@niser.ac.in}

\begin{document}


\begin{abstract}

In this article, we show that the mixed $q$-deformed Araki-Woods von Neumann algebra  $\Gamma_T(H_\R, U_t)^{\prime\prime}$, introduced  in (\cite{BKM20}), has  trivial bicentralizer, whenever it is of type $\mathrm{III}_1$.

\end{abstract}

\maketitle
\section{Introduction}
In  (\cite{BKM20}), authors associated any pair  $(\h_{\R}, U_t)$ of separable real Hilbert space 
($\dim( \h_{\R})> 1$) and strongly continuous orthogonal representation of $\R$ on $\h_{\R}$ to a von Neumann algebra coming from the non-tracial representation of the $q_{i,j}$-commutation relation for $-1 < q_{i,j}=q_{j, i} < 1$. The concerned von Neumann algebras are named as  \textit{ mixed $q$-deformed Araki-Woods von Neumann algebras}. In this article, we study Connes bicentralizer conjecture for this von Neumann. 

Let $M$ be a $\sigma$-finite von Neumann algebra equipped with a faithful normal state $\varphi$.
Then recall that the  asymptotic centralizer $AC_{\varphi}(M)$ and the 
bicentralizer $B_{\varphi}(M)$ are defined as follows:
\begin{align}
&AC_{\varphi}(M):=\set{(x_n)_n\in \ell^{\infty}(\N,M)| \lim_{n\rightarrow\infty} \norm{x_n\varphi-\varphi x_n}_{M_*} =0}; \\
&B_{\varphi}(M) := \set{a\in M | ax_n-x_na\rightarrow 0 \ \text{ultrastrongly for all}\ (x_n)_n\in AC_{\varphi}(M)}.
\end{align}
Here, the notation $x\varphi y$ for $x,y \in M$ means $x\varphi y(a)=\varphi(xay).$ The bicentralizer $B_{\varphi}(M)$ is a von Neumann subalgebra of $M$, which is globally invariant under the modular group $(\sigma_t^{\varphi})_{t\in\R}$. It follows from Connes-Størmer transitivity theorem  (see  \cite{CS78})  that if $B_{\varphi}(M)=\C1$ then $B_{\psi}(M)=\C1$ for all faithful normal state $\psi$ on $M$,  where $1$ denotes the identity operator in $M$. Hence, the triviality of the bicentralizer does not  depend on the choice of faithful normal states on $M$.

Connes classified all amenable factors  which are not of type  $\mathrm{III}_1$ in his famous paper  \cite{Connes-Class}. 
First he encountered the bicentralizer problem when he was attempting to establish the uniqueness of  amenable type $\mathrm{III}_1$ factors.
Indeed, Connes proved that an amenable type $\mathrm{III}_1$ factor $M$ is isomorphic to the unique Araki-Woods factor of type  $\mathrm{III}_1$ if and only if $B_{\varphi}(M)=\C1$  for some faithful  normal state $ {\varphi} \in M_*$. \\
Furthermore, he conjectured that  $B_{\varphi}(M)=\C1$  for every type $\mathrm{III}_1$ factor $M$  and every  faithful normal state $ {\varphi} \in M_*$. This conjecture, known as Connes' ``bicentralizer problem" and thereafter,   
Haagerup is renowned for solving this  problem when $M$ is amenable type  $\mathrm{III}_1$ factors (see \cite{Haag-Bicen}).  
Thus,  Haagerup's work  together with Connes' work, settled the classification problem for amenable factors of all types. However, when $M$ is not amenable, Connes' Bicentralizer problem is still extensively  open today. 

But  the conjecture has been supported by several  concrete examples of such factors. For example, factors with almost periodic states, free Araki-Woods factors and  semisolid factors (see \cite{Co74}, \cite{H09}, and \cite{HI17} respectively). 
In \cite[Main Theorem]{HI20},  the authors proved that $q$-Araki-Woods factors associated to  orthogonal  representations $(U_t)$ with weakly mixing component  have trivial bicentralizer. So, the bicentralizer problem for  a  $q$-deformed Araki-Woods von Neumann algebra associated to  an almost periodic orthogonal representation remains unsolved.

In this article, we study bicentralizer problem of mixed   $q$-deformed Araki-Wood von Neumann algebra.
Since from the  inception of the   theory of von Neumann  algebras,  Araki-Woods von Neumann algebras and their generalizations are  studied extensively for their rich structures. It begins with Voiculescu's Free Gaussian functor. In free probability Voiculescu associated a $C^*$-algebra $\Gamma(\h_\R)$ to real Hilbert space, generated by $s(\xi),  \xi \in \h_\R$ 
 where each $s(\xi)$  is the sum of creation and annihilation operators on
the full Fock space of the complexification of $\h_\R$. The associated von Neumann algebra $\Gamma(\h_\R)$ is isomorphic to $L( \F_{ \dim({\h_R})})$ (the group von Neumann algebra associated to the  free group of $ \dim({\h_\R})$ generators).
Further, it has the following deformation: 
\begin{enumerate}
	\item The q-Gaussian functor due to  Bo\.zejko and Speicher for $-1 < q < 1$  (see   \cite{BS}). The   associated von Neumann algebras are called q-Gaussian von Neumann algebras.
	\item The free CAR functor due to Shlyakhtenko  (see  \cite{Shly}) and the    associated von Neumann algebras are called free Araki-Woods factor. 
	\item  q-deformed Araki-Woods von Neumann algebras which is combination of $(1)$ and $(2)$,  consructed by Hiai (see \cite{Hiai}). 
\end{enumerate}
The mixed $q$-deformed Araki-Woods von Neumann algebras are further generalization of   q-deformed Araki-Woods von Neumann algebras. In recent years, there has been a lot of interest in these von Neumann algebras.
The reader may look at \cite{BKM20}, \cite{BM17}, \cite{BMRW}, \cite{R05}, \cite{MK1}, \cite{W21}, and the references therein for an overview of the literature.

The mixed $q$-deformed Araki-Woods von Neumann algebras are relatively new in the area.
So, far its factoriality, type classification, Haagerup property, non-injectivity properties  are studied (see \cite{BKM20}, \cite{BKM21}). 
In  \cite{BKM23} and \cite{BKM20} factoriality of these algebras  are studied under different conditions on $(U_t)_{t\in\R}$.
Recently, in  \cite{MK1}, the factoriality  is studied  when the underlying Hilbert space is finite dimension. By combining the results of   \cite{BKM23} and \cite{MK1},  the  type classification question can be  answered in complete  generality.

To study the bicentralizer, we use the  ultraproduct description of bicentralizer as given in 
\cite[Proposition.3.3]{HI17}. Suppose $M$ is a  von Neumann algebra with a faithful normal state $\varphi$, then we write $M^\varphi $ for the centralizer of $M$ with respect to $\varphi$ and   $M^\varphi  = \{ x\in M :  \varphi(xy)= \varphi(yx) \text{ for all } y \in M \}$. It  is shown (see \cite[Proposition.3.3]{HI17}) that $B_{\varphi}(M)$ can be realized as a subalgebra of the ultraproduct  of the von Neumann algebra $M$, more precisely, if $\omega$ is any free ultrafilter over   $\N$, then it follows that 
\begin{align}\label{rela}
B_\phi(M) = ((M^\omega)^{\varphi^{\omega} })' \cap M \subseteq M^\omega.
\end{align}
Then using this, we  directly show that mixed $q$-deformed Araki-Woods von Neumann algebra associated with $(U_t)$ having almost periodic component of dimension $\geq 3$, have trivial bicentralizer.   

For the other case,  we follow similar path as followed by the author in \cite{HI20}. 
But we like to point out that   due to recent development of factoriality of these von Neumann algebras,  our argument  is much simpler than \cite{HI20}. In fact finding a sequence of unitary converging to $0$ weakly in the centralizer of $M_T$ makes the argument simple.   
Now we state our main theorem: 
\begin{thm}\label{main-theorem}
	Let  $-1<q_{ij}=q_{ji}<1$ be  real numbers such that $\sup_{i,j}\abs{q_{ij}}<1$ and let $(\h_{\R}, U_t)$ be a strongly continuous orthogonal representation such that the mixed $q$-Araki-Woods algebra $M_T$  is a type $\mathrm{III}_1$ factor, then  it  has  trivial bicentralizer.
\end{thm}

We remark  that  the mixed $q$-deformed Araki-Woods von Neumann algebra $M_T$ is of type $\mathrm{III}_1$, whenever the group generated by the eigenvalues of the analytic generator $A$ of $(U_t)_{t\in\R}$ is equal to $\R_{*}^{\times}$ (see \cite{MK1} and \cite{BKM23}). Thus, $M_T$ associated to almost periodic orthogonal representation  could  be type  $\mathrm{III}_1$.  We note that  if $ q_{ i, j} = q $ for all $ i, j$, then we obtain $q$-deformed Araki-Wood von Neumann algebras. Thus, our theorem also completely solves the bicentralizer problem for  $q$-deformed Araki-Woods von Neumann algebras.

We organize this article as follows: In section 2, we describe the construction of the mixed $q$-deformed Araki-Woods von Neumann algebras and recall some basics on ultraproduct of von Neumann algebras. In section 3, we discuss bicentralizer of mixed $q$-deformed algebra  when associated orthogonal representation  with non trivial almost periodic part and in the final section we prove Theorem \ref{main-theorem}.


\section{Preliminary}
\subsection{Mixed $q$-deformed Araki-Wood von Neumann algebra}
We briefly recall the mixed $q$-deformed Araki-Woods algebra from \cite{BKM20}. Let $\h_{\R}$ be a separable real Hilbert space. 
and  $\R\ni t\mapsto U_{t} $ be a strongly continuous orthogonal representation of $\R$ on $\h_{\R}$. We denote the complexification of $\h_{\R}$ by $\h_{\C}=\h_{\R}\otimes \C$.  Denote the inner product and norm on $\h_{\C}$ by $\inner{\cdot}{\cdot}_{\h_{\C}}$ and $\norm{\cdot}_{\h_{\C}}$ respectively. All inner products in this article are considered to be linear in second variable.  Identify $\h_{\R}$ in $\h_{\C}$ by $\h_{\R}\otimes 1$. Since $\h_{\C}= \h_{\R}+i \h_{\R}$, as a real Hilbert space the inner product of $\h_{\R}$ in $\h_{\C}$ is given by $\mathfrak{R}\inner{\cdot}{\cdot}_{\h_\C}$.  We denote the complex conjugation on $\h_{\C}$ by $\J$, which is a bounded anti-linear operator.

We can extend the representation $\R\ni t \mapsto U_{t}$ from $ \h_{\R}$ to a strongly continuous one-parameter group of unitaries on $\h_{\C}$, which is denoted by $(U_t)_{t\in\R}$ with a slight abuse of notation. Let $A$ be the analytic generator of $(U_{t})_{t\in \R}$ acting on $\h_{\C}$ obtained via Stone's theorem, which is non-degenerate, positive, and self-adjoint.
Now we quickly recall the following definition: 
\begin{defn} 
\begin{enumerate}
	\item $(U_t)$  is called  almost periodic if $A$ is diagonalizable and  
	\item $(U_t)$  is called  weakly mixing $A$ has no  point spectrum.  
\end{enumerate}
 Then we note that $( \h_R, U_t)$  has a unique decomposition  $ \h_\R =  \h_\R^{ap} + \h_\R^{wm} $  where $(U_t)$  acts 
on $\h_\R^{ap}$  as an almost periodic action and on $\h_\R^{wm} $  as a weakly mixing action. 	
\end{defn}
With the help of the generator $A$, consider a new inner product $\langle \cdot,\cdot\rangle_{U}$ on $\h_{\C}$   as follows:
\begin{align*}
\inner{\xi}{\eta}_{U}= \inner{2(1+A^{-1})^{-1}\xi}{\eta}_{\h_{\C}}, \text{ for }\xi,\eta \in \h_{\C}.
\end{align*}
Let $\h$ be the completion of $\h_{\C}$ with respect to $\inner{\cdot}{\cdot}_{U}$. We denote the inner product and norm on $\h$ by $\inner{\cdot}{\cdot}_U$ and $\norm{\cdot}_{U}$ respectively.\\

Let $ r \in \N$ and $[r] =\{1,2,\dots,r\}$. Suppose  $N$ is either   $[r]$  or  $\N$. Fix a decomposition of $\h_{\R}$ as follows:
\begin{align*}
\h_{\R}:=\bigoplus\limits_{i\in N}\h_{\R}^{(i)}, 
\end{align*}
where $\h_{\R}^{(i)}$, $i\in N$, are non-trivial invariant subspaces of $(U_{t})_{t\in\R}$. Fix some real numbers $q_{ij}$ such that, $-1<q_{ij}=q_{ji}<1$ for $i,j\in N$ with $\sup\limits_{ i,j \in N} \abs{q_{ij}}<1 $. For fix $i,j\in N$, we define $T_{i,j}:\h_{\R}^{(i)}\otimes \h_{\R}^{(j)}\rightarrow \h_{\R}^{(j)}\otimes \h_{\R}^{(i)}$ to be the bounded extension of:
\begin{align*}
 \xi\otimes \eta\mapsto q_{ij} (\eta \otimes \xi), \text{ for }\xi\in \h_{\R}^{(i)},\eta \in \h_{\R}^{(j)}.
\end{align*}
Then, $T_\R:=\oplus_{i,j\in N} T_{i,j}\in \B(\h_\R\otimes\h_\R)$. Note that we can extend $T_\R$ to $\h_{\C}\otimes \h_{\C}$ linearly and denote the extension by $T_\C$.
By a simple density argument, it follows that $T_\C$ admits a unique bounded extension $T$ to $\h\otimes \h$ such that $T(\h^{(i)}\otimes\h^{(j)})\subset \h^{(j)}\otimes\h^{(i)}$, where $\h^{(i)}$'s are the completion of $\h_{\C}^{(i)}$ (for $i\in N$) with respect to the  deformed inner product $\inner{\cdot}{\cdot}_{U}$. It is easy to verify that $T:=\oplus_{i,j\in N}T_{ij}$, where $T_{ij}:\h^{(i)}\otimes\h^{(j)}\rightarrow\h^{(j)}\otimes\h^{(i)}$ is defined as the extension of the map:
\begin{align}\label{T action}
\xi\otimes \eta\mapsto q_{ij} (\eta \otimes \xi), \text{ for }\xi\in \h^{(i)},\eta \in \h^{(j)}.
\end{align}

Moreover, $T$ has the following properties:
\begin{align}\label{Eq.T}
&T^{*}=T,\quad\quad\quad\quad\quad(\text{since } q_{ij}=q_{ji}, \text{ for }i,j\in N); \\
\nonumber&\norm T_{\h\otimes\h}< 1,\,\,\,\quad\quad (\text{since}\sup_{i,j\in   N} \abs{q_{ij}}<1);\\
\nonumber&(1\otimes T)(T\otimes 1)(1\otimes T)=(T\otimes1)(1\otimes T)(T\otimes 1),
\end{align}
where $1\otimes T \text{ and }T\otimes 1$ are the natural amplifications of $T$ to $\h\otimes \h\otimes\h$. The third relation listed in Eq. \eqref{Eq.T} is referred as  Yang-Baxter equation.

Let $\F(\h):=\C\Omega\oplus\bigoplus\limits_{n=1}^{\infty}\h^{\otimes n}$ be the full Fock space of $\h$, where $\Omega$ is a distinguished unit vector $($vacuum vector$)$ in $\C$. By convention, $\h^{\otimes 0}:=\C\Omega$. The canonical inner product and norm on $\F(\h)$ will be denoted by $\inner{\cdot}{\cdot}_{\F(\h)}$ and $\norm{\cdot}_{\F(\h)}$ respectively. 
For $\xi\in \h$, let $a(\xi)$ and $a^{*}(\xi)$ denote the canonical left creation and annihilation operators acting on $\F(\h)$ which are defined as follows:
\begin{align}\label{gen op1}
&a(\xi)\Omega =\xi,\\
\nonumber& a(\xi)(\xi_{1}\otimes \xi_{2} \otimes\cdots \otimes \xi_{n})=\xi \otimes\xi_{1}\otimes\xi_{2}\otimes\cdots \otimes \xi_{n}\\ \text{ and},\ \ \
\nonumber& a^{*}(\xi)\Omega=0,\\
\nonumber&a^{*}(\xi)(\xi_{1}\otimes \xi_{2} \otimes\cdots \otimes \xi_{n})=\langle \xi ,\xi_{1}\rangle_{U} \xi_{2}\otimes\cdots \otimes \xi_{n},
\end{align}
where $\xi_{1}\otimes\cdots\otimes\xi_{n}\in\h^{\odot n}$ $(\h^{\odot n}$ denotes the $n$-fold algebraic tensor product of $\h)$ for $n\geq 1$. The operators $a(\xi)$ and $a^{*}(\xi)$ are bounded and adjoints of each other on $\F(\h)$.
Let $T_{i}$ be the operator acting on $\h^{\otimes(i+1)}$ for $i \in \N$ defined as follows:
\begin{align}\label{ExtndTDefn}
T_{i}:=\underset{i-1}{\underbrace{1\otimes\cdots\otimes 1}} \otimes T.
\end{align}
Extend $T_{i}$ to $\h^{\otimes n}$ for all $n>i+1$ by $T_{i}\otimes\underset{n-i-1}{\underbrace{1\otimes\cdots\otimes1}}$ and denote the extension again by $T_i$ with a slight abuse of notation.
Let $S_{n}$ denote the symmetric group of $n$ elements. Note that $S_1$ is trivial. For $n\geq 2$, let $\tau_{i}$ be the transposition between $i$ and $i+1$. It is  known that the set $\lbrace\tau_{i}\rbrace_{i=1}^{n-1}$  generates  $S_{n}$. 
For $n\in \N$, let $\pi:S_{n}\rightarrow\B(\h^{\otimes n})$ be the quasi-multiplicative map given by $\pi(1)=1$ and $\pi(\tau_{i})= T_{i}\text{ } (i=1,\dots,n-1)$. Consider $P^{(n)}\in\B(\h^{\otimes n})$, defined as follows:
\begin{align}\label{Twisted inner product}
P^{(n)}:=\underset{\sigma\in S_{n}}\sum \pi(\sigma).
\end{align}
By convention, $P^{(0)}$ on $\h^{\otimes0}$ is identity.
From the properties of $T$ in Eq. \eqref{Eq.T} and \cite[Theorem 2.3]{BSp94}, it follows that $P^{(n)}$ is a strictly positive operator for every $n\in \N$. Following \cite{BSp94}, the association 
\begin{align}\label{New Fock inner pdt}
\inner{\xi}{\eta}_{T}=\delta_{n,m}\inner{\xi}{P^{(n)}\eta}_{\F(\h)},\text{ for }\xi\in \h^{\otimes m},\eta\in \h^{\otimes n},
\end{align}
defines a definite sesquilinear form on $\F(\h)$ and let $\F_{T}(\h)$ denote the completion of $\F(\h)$ with respect to the norm on $\F(\h)$ induced by $\inner{\cdot}{\cdot}_T$. We denote the inner product and the norm on $\F_T(\h)$ by $\inner{\cdot}{\cdot}_T$ and $\norm{\cdot}_T$ respectively. We also denote $\F^{\text{finite}}_{T}(\h):=\text{ span }_\C \{\h^{\otimes n}, n\geq0\rbrace$ and $\h^{\otimes_T^n} =\overline{\h^{\otimes n}}^{\norm{\cdot}_T}$ for $n\in \N$.\\
Note that $\overline{\h}^{\norm{\cdot}_T}=\h$. Following \cite{BSp94}, for $\xi\in \h$, consider the $T$-deformed left creation and annihilation operators on $\mathcal{F}_T(\h)$ defined as follows: 
\begin{align}\label{gen op2}
l(\xi)&:=a(\xi),\text{ and, }\\
\nonumber l^{*}(\xi)&:=\begin{cases}a^{*}(\xi)(1+T_{1}+T_{1}T_{2}+\cdots +T_{1}T_{2}\cdots T_{n-1}), \text{ on }\h^{\otimes n};\\
0, \quad\text{ on }\C\Omega.\end{cases}
\end{align}
Then, $l(\xi)$ and $l^{*}(\xi)$ admit bounded extensions to $\mathcal{F}_T(\h)$ with $\norm{T}_{\h\otimes\h} = q < 1$, we have\\
\begin{align}
\norm{l(\xi)}\leq\norm{\xi}_{U}(1-q)^{-\frac{1}{2}},\\
\nonumber \norm{l^{*}(\xi)}\leq\norm{\xi}_{U}(1-q)^{-\frac{1}{2}}.
\end{align}

As the definition  of  $l^{*}(\xi)$   involves the operator $T$ (see Eq. \eqref{gen op2}), so, using the Eq. \eqref{T action}, the action of $l^{*}(\xi)$ on $\mathcal{F}_{T}(\h)$ can be described further. Indeed, let  $n\in \N$ and $\xi_{i_{k}}\in\h^{(i_{k})}$ for $i_{k}\in N$ and $1\leq k\leq n$. Then,
\begin{align}\label{left annihilation}
&l^{*}(\xi)\Omega=0, \text{ and, }\\ 
\nonumber&l^{*}(\xi)(\xi_{i_{1}}\otimes\cdots\otimes \xi_{i_{n}})=\underset{k=1}{ \overset{n}\sum}\langle \xi,\xi_{i_{k}}\rangle_{U}q_{i_{k}i_{k-1}}\cdots q_{i_{k}i_{1}}(\xi_{i_{1}}\otimes \xi_{i_{2}}\otimes\cdots\\
\nonumber&\quad\quad\quad\quad\quad\quad\quad\quad\quad\quad\quad\quad\quad\quad\quad\cdots \otimes \xi_{i_{k-1}}\otimes
\xi_{i_{k+1}}\otimes\cdots\otimes \xi_{i_{n}}).
\end{align}

$$ $$

Define 
\begin{align}\label{genop}
s(\xi):= l(\xi)+l^*(\xi), \text{ for }\xi \in \h_{\R}.
\end{align}

We denote the $C^*$-algebra generated by the self-adjoint operators $\{s(\xi):\xi\in\h_\R\}$ by  $\Gamma_{T}(\h_{\R},U_{t})$ and  $\Gamma_{T}(\h_{\R},U_{t})^{\prime\prime}\subseteq \mathbf{B}(\mathcal{F}_T(\h))$ becomes  the associated von Neumann algebra.  The von Neumann algebra  $\Gamma_{T}(\h_{\R},U_{t})^{\prime\prime}$ is called   \textit{mixed $q$-deformed Araki-Woods von Neumann algebra}. This von Neumann algebra is equipped with a vacuum state $\varphi$ given by $\varphi( \cdot ) = \varphi_T(\cdot)=\inner{\Omega}{\cdot\Omega}_T$. The vacuum vector $\Omega$ is a cyclic and separating vector for $\Gamma_{T}(\h_{\R},U_{t})^{\prime\prime}$. Also $\Gamma_{T}(\h_{\R},U_{t})^{\prime\prime}$ is in standard form as an algebra acting on $\mathcal{F}_T(\h)$. For the  simplicity of notation  we write $$M_T = M_T(\h_{\R}, U) = \Gamma_{T}(\h_{\R},U_{t})^{\prime\prime}.$$ 
If $\D_{\R} $ is  a $(U_t)$-invariant subspace of $\h_{\R}$, then we write 
$$M_T(\D_{\R}, U) = \Gamma_{T}(\D_{\R},U_{t})^{\prime\prime}.$$


For $\xi\in \h$, let $r(\xi)$ and $r^{*}(\xi)$ denote the deformed right creation and annihilation operators  acting on $\mathcal{F}_{T}(\h)$ and are defined as follows:
\begin{equation}\label{def b}
\begin{aligned}
&r(\xi)\Omega=\xi,
\\&r(\xi)(\xi_{i_1}\otimes \cdots \otimes \xi_{i_n})=\xi_{i_1}\otimes\cdots \otimes \xi_{i_n}\otimes \xi,\text{ and,}\\
&r^*(\xi)\Omega=0,\\ 
\nonumber &r^*(\xi)(\xi_{i_{1}}\otimes \cdots \otimes \xi_{i_{n}})=\underset{k=1}{ \overset{n}\sum}\langle \xi,\xi_{i_k}\rangle_{U}q_{i_ki_{k+1}}\cdots q_{i_{k}i_{n}}(\xi_{i_{1}}\otimes  \xi_{i_{2}}\otimes \cdots \otimes \xi_{i_{k-1}}\otimes \xi_{i_{k+1}}\otimes \cdots \otimes  \xi_{i_n}), 
\end{aligned}
\end{equation}
where $n\in\N$ and $\xi_{i_{k}}\in\h^{(i_{k})}$ for $i_{k}\in N$, $1\leqslant k\leqslant n$.
We write $$ d(\xi): = r(\xi) + r(\xi)^* \text{ for } ~\xi \in \h.$$
For $j\in N$, let
$(\h^{(j)}_{\R})^{\prime}=\lbrace \xi\in \h^{(j)}:\langle \xi,\eta\rangle_{U}\in \R \text{ for all } \eta \in \h_{\R}\rbrace$.
Note that $(\h^{(j)}_{\R})^{\prime}$ is a real subspace of $\h$ and by Hahn-Hellinger theorem it follows that $$\overline{(\h^{(j)}_{\R})^{\prime}+i(\h^{(j)}_{\R})^{\prime}}^{\norm{\cdot}_{U}}=\h^{(j)}.$$ 
Define $ \h^{\prime}_{\R}=\bigoplus \limits_{j\in N}(\h^{(j)}_{\R})^{\prime}$, then we have
\begin{align}\label{complement}
\h= \overline{ \h^{\prime}_{\R} + i \h^{\prime}_{\R} }^{\norm{\cdot}_{U}}  .
\end{align}
We recall the following results from \cite[see Section 3.2 and Theorem 3.14]{BKM20} . 
\begin{thm}\label{commutant thm}
	Let $\xi\in \mathfrak{D}(A^{-1})\cap\h_{\R}$, where $\mathfrak{D}(A^{-1})$ denotes the domain of $A^{-1}$. Then,
	\begin{enumerate}
		\item $ A^{-\frac{1}{2}}\xi \in \h_\R'$ and 
		\item $J_{\varphi}s(\xi)J_{\varphi}=d(A^{-\frac{1}{2}}\xi),$ where $J_{\varphi}$ is the modular conjugation operator.
	\end{enumerate}
	Moreover, $\Gamma_{T}(\h_{\R},U_{t})^{\prime}=\lbrace d(\xi):\xi \in \h^{\prime}_{\R}\rbrace^{\prime\prime}.$
\end{thm}
Since $\Omega$ is cyclic and separating for $\Gamma_{T}(\h_{\R},U_{t})^{\prime\prime}$  (see \cite[Proposition. 3.5]{BKM20}), for $\zeta\in \Gamma_{T}(\h_{\R},U_{t})^{\prime\prime}\Omega$ and $\eta \in \Gamma_{T}(\h_{\R},U_{t})^{\prime}\Omega$, there exists a unique operators $s(\zeta)$ in $\Gamma_{T}(\h_{\R},U_{t})^{\prime\prime}$ and $d(\eta)$ in $\Gamma_{T}(\h_{\R},U_{t})^{\prime}$ such that $s(\zeta)\Omega=\zeta$ and $d(\eta)\Omega=\eta$. The Wick product formula in \cite[Lemma 4.1]{BKM20} provides a natural  expression for $s(\xi_{t_1}\otimes \cdots \otimes \xi_{t_n})$ $($resp. $d(\eta_{t_1} \otimes\cdots \otimes \eta_{t_n}))$ for $\xi_{t_{k}}\in\h_{\C}^{(t_{k})}$ $($resp. $\eta_{t_{k}}\in (\h^{(t_{k})}_{\R})^{\prime}+i(\h^{(t_{k})}_{\R})^{\prime})$, $t_{k}\in N$, $1\leqslant k\leqslant n$, in terms of the operators $l(\xi_{t_{k}})$ and  $l^{*}(\xi_{t_{k}})$ $($resp.  $r(\eta_{t_{k}})$ and $r^{*}(\eta_{t_{k}}))$ as follows. 

\begin{lem}\cite[Lemma 4.1]{BKM20}\label{Wick pdt formula}
	Fix $n\in \N$. Let $\xi_{t_{k}}\in\h_{\C}^{(t_{k})}$, $\eta_{t_{k}}\in(\h^{(t_{k})}_{\R})^{\prime}+i(\h^{(t_{k})}_{\R})^{\prime}$ for $t_{k}\in N$, $1\leqslant k\leqslant n$. Then,
	\begin{align*}
	(i)\quad s(\xi_{t_1}\otimes \cdots \otimes \xi_{t_n})=\sum_{\substack{ l,m\geqslant0 \\ l+m=n}}\sum_{\substack{I=\lbrace i(1),\cdots ,i(l) \rbrace,\, i(1)< \cdots < i(l)\\ 
			J=\lbrace j(1),\cdots,j(m)\rbrace,\,j(1)< \cdots < j(m)\\ I\cup J=\lbrace 1,\cdots ,n\rbrace\\I\cap J=\emptyset}}f_{(I,J)}(q_{ij})l(\xi_{t_{i{(1)}}})\cdots l(\xi_{t_{i{(l)}}})\\
	\nonumber\quad\quad l^*(\mathcal{J}\xi_{t_{j{(1)}}})\cdots l^*(\mathcal{J}\xi_{t_{j{(m)}}}),\\
	(ii)\quad d(\eta_{t_1}\otimes \cdots \otimes \eta_{t_n})
	=\sum_{\substack{ l,m\geqslant0 \\ l+m=n}}\sum_{\substack{I=\lbrace i(1),\cdots ,i(l) \rbrace,\, i(1)< \cdots < i(l)\\
			J=\lbrace j(1),\cdots,j(m)\rbrace,\,j(1)< \cdots < j(m)\\
			I\cup J=\lbrace 1,\cdots ,n\rbrace\\
			I\cap J=\emptyset}}\widetilde{f}_{(I,J)}(q_{ij})r(\eta_{t_{i{(1)}}})\cdots r(\eta_{t_{i{(l)}}})\\
	\nonumber \quad\quad\quad\quad\quad\quad\quad\quad\quad\quad\quad\quad\quad\quad\quad\quad\quad\quad r^*(\mathcal{J}_{r}(\eta_{t_{j{(1)}}}))\cdots r^*(\mathcal{J}_{r}(\eta_{t_{j(m)}})),
	\end{align*}
	where $f_{(I,J)}(q_{ij})=\underset{\lbrace (r,s):\text{ }1\leqslant r\leqslant l,1\leqslant s\leqslant m,\text{ } i(r)>j(s)\rbrace}\prod  q_{t_{i(r)}t_{j(s)}}, $\\
	$\quad \quad\widetilde{f}_{(I,J)}(q_{ij})= \underset{\lbrace (r,s):\text{ }1\leqslant r\leqslant l,1\leqslant s\leqslant m,\text{ } i(r)<j(s)\rbrace}\prod  q_{t_{i(r)}t_{j(s)}}$, $\mathcal{J}$ and $\mathcal{J}_{r}$ denote the complex conjugations  on $\h_{\C}$ and $\h^{\prime}_{\R}
	+i\h^{\prime}_{\R}$ respectively.
\end{lem}

\subsection{Ultraproducts of von Neumann algebras}\label{ultra}
We require some basics of ultraproduct of Hilbert spaces and von Neumann algebras for our analysis of bicentralizer. Therefore, in this subsection, we recall the notions of ultraproducts of von Neumann algebras from \cite{AH14}. Originally the ultraproducts of von Neumann algebras was defined for tracial von Neumann algebras. In the case of type III von Neumann algebras, there are two notions of ultraproducts. One is given by Ocneanu and other one  by Groh- Raynaud. Ocneanu's ultraproduct is a generalisation  of ultraproducts of tracial von Neumann algebras. In this article use the 
Ocneanu's ultraproducts of von Neumann algebras.\\

Before proceeding to the ultraproduct of von Neumann algebras, we briefly recall the ultraproduct of Banach spaces. Fix a free ultrafilter $\omega$ on $\N$. Recall that, $\omega\in\beta\N\setminus\N$, where $\beta\N$ is the Stone-Cech compactification of $\N$. Let $\ell^{\infty}(\N,E_n)$ be the set of all sequences $(x_n)_n\in \displaystyle\prod_{n\in\N} E_n$ with $\sup_{n\in\N}\norm{x_n}< \infty$. Clearly, $\ell^{\infty}(\N,E_n)$ becomes a Banach space  with the norm given by $\norm{(x_n)_n} = \sup_{n\in\N}\norm{x_n}$, for $(x_n)_n\in\ell^{\infty}(\N,E_n)$. 
Let $\mathcal{I}_{\omega}$ denote the closed subspace of all $(x_n)_n\in\ell^{\infty}(\N,E_n)$ which satisfies
$\lim_{n\rightarrow\omega}\norm{x_n} = 0.$ Then the Banach space ultraproduct $(E_n)_{\omega}$ is defined as the quotient space $\ell^{\infty}(\N,E_n)/I_{\omega}$. Any element of $(E_n)_{\omega}$ represented by $(x_n)_n\in\ell^{\infty}(\N,E_n)$ is written as $(x_n)_{\omega}$. If $(H_n)_{n\in\N}$ is a sequence of Hilbert spaces, then the Banach space ultraproduct $H_{\omega}:=(H_n)_{\omega}$ is again a Hilbert space with the inner product defined by $\inner{(\xi_n)_{\omega}}{(\eta_n)_{\omega}}:= \lim_{n\rightarrow\omega}\inner{\xi_n}{\eta_n}$, for $(\xi_n)_{\omega},(\eta_n)_{\omega}\in (H_n)_{\omega}$.

We first introduce the Ocneanu ultraproduct of a family of von Neumann algebras with normal faithful states $(M_n,\varphi_n)$ along the free ultrafilter $\omega$  on $\N$.  Consider the spaces 
\begin{align*}
\ell^{\infty}(\N,M_n):=\set{(x_n)_n\in\prod_{n\in\N}M_n:\ \sup_{n}\norm{x_n}<\infty},
\end{align*}
and 
\begin{align*}
\I_{\omega}:=\set{(x_n)_n\in\ell^{\infty}(\N,M_n):\ \lim_{n\rightarrow\omega}\norm{x_n}^{\sharp}=0},
\end{align*}
where $\norm{x_n}^{\sharp}=\varphi_n(x_n^{*}x_n+x_nx_n^{*})^{1/2}$. Define 
\begin{align*}
\M^{\omega}(M_n,\varphi_n):=\set{(x_n)_n\in\ell^{\infty}(\N,M_n):\ (x_n)_n\I_{\omega}\subset\I_{\omega},\ \text{and}\ \I_{\omega}(x_n)_n\subset\I_{\omega}}.
\end{align*}
Note that, $\M^{\omega}(M_n,\varphi_n)$ is $C^*$-algebra in which $\I_{\omega}$ is a closed (two-sided) ideal. Hence, we can consider the quotient $C^*$-algbera
\begin{align*}
(M_n,\varphi_n)^{\omega}:=\M^{\omega}(M_n,\varphi_n)/\I_{\omega}.
\end{align*}
By the same proof as in \cite[\textsection 5.1]{Oc85}, $(M_n,\varphi_n)^{\omega}$ turns out be a von Neumann algebra, and called the Ocneanu ultraproduct of the sequence $(M_n,\varphi_n)_{n\in\N}$. If  all the $ (M_n, \varphi_n)$ are  $(M, \varphi) $ ( a von Neumann $M$ with faithful normal state $\phi$),  then we write 
$$ M^\omega :=\M^{\omega}(M,\varphi)/\I_{\omega}.$$
The equivalence class of a sequence $(x_n)_{n\in\N}\in\M^{\omega}(M_n,\varphi_n)$ in the quotient algebra is denoted by $(x_n)^{\omega}$. The following defines a normal faithful state $(\varphi_n)^{\omega}$ on $(M_n,\varphi_n)^{\omega}$: 
\begin{align*}
(\varphi_n)^{\omega}((x_n)^{\omega}):=\lim_{n\rightarrow\omega}\varphi_n(x_n), \quad\text{for}\quad (x_n)^{\omega}\in(M_n,\varphi_n)^{\omega}.\\
\end{align*}

For the Groh-Raynaud ultraproduct, suppose 
 $(M_n,\varphi_n)$ is a sequence of von Neumann algebras equipped with normal faithful states. For each $n\in\N$, we can realize $M_n\subset B(H_n)$, using GNS representation of $M_n$ with respect to $\varphi_n$.  Let $(M_n)_{\omega}$ be the Banach space ultraproduct of $(M_n)_{n\in\N}$. Recall that, $(M_n)_{\omega}$ is $C^*$-algebra with respect to the norm $\norm{(x_n)_{\omega}}=\displaystyle\lim_{n\rightarrow\omega}\norm{x_n}$. Also, let $H_{\omega}$ be the Banach space ultraproduct of the sequence $(H_n)_{n\in\N}$. Consider the  diagonal action $\pi_{\omega}:(M_n)_{\omega}\rightarrow B(H_{\omega})$ defined by 
\begin{equation}
\pi_{\omega}((a_n)_{\omega})(\xi_n)_{\omega}:=(a_n\xi_n)_{\omega}.
\end{equation}
It can be easily checked that, this action is a well-defined $*$-homomorphism and 
\begin{equation*}
\norm{\pi_{\omega}((a_n)_{\omega})(\xi_n)_{\omega}}=\displaystyle\lim_{n\rightarrow\omega}\norm{a_n}=\norm{(a_n)_{\omega}}. 
\end{equation*}
Hence, $\pi_{\omega}$ is an injective $*$-homomorphism. The Groh- Raynaud ultraproduct is defined as the weak closure of $\pi_{\omega}((M_n)_{\omega})$ inside $B(H_{\omega})$ and  denoted by $\prod^{\omega}(M_n,\varphi_n)$. The ultraproduct state on the Groh- Raynaud  ultraproduct, denoted by $\varphi_{\omega}$, is a vector state given by the ultraproduct of the cyclic vectors for the GNS representations of algberas $M_n$, i.e.
\begin{equation} \varphi_{\omega}(x):=\inner{\xi_{\omega}}{x\xi_{\omega}},\quad \text{for}\quad x\in\prod^{\omega}(M_n,\varphi_n),
\end{equation}
 where $\xi_{\omega}:=(\xi_n)_{\omega}\in(\h_n)_{\omega}$.

\section{ Bicentralizer of $M_T$ when $\h_{\R}^{ap} \neq 0$}
In this section, we show that  bicentralizer of $M_T$ is trivial when it is type  $\mathrm{III}_1$  and $\h_{\R}^{ap} \neq 0$. 
Suppose  $\EssD_{\R}\subset\h_{\R}$  is  $(U_t )_{t\in\R}$-invariant subspace, then we consider  $M_T (\EssD_{\R}, U ) := \Gamma_T (\D_{\R} , U_t )^{\prime\prime}$. We begin with the following  lemma which will help our analysis in the sequel.


\begin{lem}\label{unitary}
	Suppose  $(\h_{\R}, U_t)$ has non-trivial almost periodic part, i.e, $\h_{\R}^{ap} \neq 0$ and $ \EssD_\R \subseteq 	\h_{\R}^{ap}$ be a   $(U_t )$-invariant subspace.  Then there exists a sequence of unitaries 
 $(u_n) \subseteq  M_T(\EssD_{\R} ) \cap M_T^\varphi $ such that $ u_n \rightarrow 0$ as $n \rightarrow \infty$  in w.o.t. 	
\end{lem}

\begin{proof}
We first assume that 	$ \text{dim}(\EssD_\R ) = 1$, so, there exists a vector $ \xi_0 \in \D_{\R}$ such that $ \D_{\R} = \R \xi_0 $ and $U_t \xi_0 = \xi_0$ for all $ t \in \R$. Then it follows from  \cite[ Theorem 4.11]{BKM20} that  $M_T(\EssD_\R  )$ is   a diffuse abelian von Neumann algebra and since $U_t \xi_0 = \xi_0$ for all $ t \in \R$, so,   $M_T(\EssD_\R  ) \subseteq  M_T^\varphi $. Hence, in this case  such sequence always exists.

Now suppose  	$ \text{dim}(\EssD_\R ) \geq 2$. 
Then it is established  (see \cite{BKM23}, \cite{MK1}) that   $M_T(\EssD_\R  )$  is  a non-type  $\mathrm{I}$ factor  and it is a diffuse von Neumann algebra and 	$( M_T(\EssD_\R  )^\varphi $ is also a factor. Indeed, it is a  $\mathrm{II}_1$ factor (see \cite{BKM23}, \cite{MK1}).
So, it is a diffuse von Neumann algebra, thus there exists a sequence of unitaries  $(u_n) \subseteq ( M_T(\EssD_{\R} ))^\varphi $ such that $ u_n \rightarrow 0$ as $n \rightarrow \infty$  in w.o.t and since  $( M_T(\EssD_{\R}) )^\varphi \subseteq M_T ^\varphi$, so $(u_n) \subseteq M_T^\varphi$.
\end{proof}

The following lemma will be crucial for our  analysis.
\begin{lem}\label{inva}
Let $ \D_\R \subseteq \h_\R $ be  a $(U_t)$-invariant subspace. Suppose 
$\D_\C$ is the complexification of $\D_{\R}$ and consider $\D= \overline{\D_\C}^{ \norm{ \cdot}_U} $.
Then $$\F_T(\D)^\perp  = \overline{\text{span}\{  \xi_1\otimes\cdots\otimes\xi_n: ~ n\in \N,   \xi_i \in \h_\R' \text{ for all } i \in [n] \text{ and } \xi_j \in \D^\perp \text{ for some } j \in [n]  \}}^{ \norm{ \cdot}_U},$$
where $[n] = \{ 1, 2, \cdots, n\}$.
\end{lem}
\begin{proof}
Since $\D_\R$ is $(U_t)$-invariant, so, we note that  
\begin{align*}
\D = \overline{ \D \cap (\h_\R' + i \h_\R')}^{ \norm{ \cdot}_U},~~ (\text{see Eq. } \ref{complement}). 
 \end{align*}
The rest follows from standard arguments.

\end{proof}

\begin{rem}\label{condi}
Let $ \D_\R \subseteq \h_\R $ be  a $(U_t)$-invariant subspace. Then  there exists a faithful normal conditional expectation   $E : M_T\rightarrow M_T(\D_{\R}, U_t)$, which preserve the state $\varphi$. Indeed, let  $ (\sigma_t)$ be  the modular automorphism group of $M_T$ with respect to the vacuum state $\varphi$. As $ \D_{\R}$ is $(U_t)$-invariant, so, form  \cite[Section 3.1]{BKM20}, it follows that $$ \sigma_t( M_T( \D_{\R}, U)) \subseteq M_T( \D_{\R}, U) \text{ for all } t \in \R.$$
Therefore, by Takesaki's theorem  \cite{TMCon}, such conditional expectation always exists.

\end{rem}

\begin{lem}\label{commu}
Let $ \D_\R \subseteq \h_\R $ be  a $(U_t)$-invariant subspace  and 
 $u=(u_n) \subseteq  M_T(\D_\R, U) \cap M_T ^\varphi $   a sequence  of unitaries  such that $ u_n \rightarrow 0$ as $n \rightarrow \infty$  in w.o.t.  
 Suppose
$x \in M_T$ with  $xu = ux $, then $ x \in M_T( \D_\R, U) $.  Furthermore, 
$$  M_T( \D_\R, U_t) \big)' \cap M_T \subseteq M_T( \D_\R, U).$$
\end{lem}

\begin{proof}
Let  $E : M_T\rightarrow M_T(\D_{\R}, U_t)$ be the  $\varphi$-preserving faithful normal conditional expectation (see Remark \ref{condi}). 
Suppose $x_0 := x-E(x).$ Fix any $\xi\in F_T(\D)^{\perp},$ where $\D$ is defined in Lemma \ref{inva}, which is of the form $\xi=\xi_1\otimes\cdots\otimes\xi_n,$ for some $n\in\N,$
	where one of the  $\xi_k$ is contained in $\D^\perp \cap \h_\R'$.  Then we wish to prove  the following 
	$$ \langle \xi, x_0\Omega \rangle = 0.$$ 
First it follows from Lemma \ref{commutant thm} that $d(\xi) =d( \xi_1\otimes \cdots \otimes \xi_n) \in M_T'$. Hence, we have  
\begin{align*}
 \langle \xi, x_0\Omega \rangle &=  \langle d(\xi) \Omega, x_0\Omega \rangle 	\\
  &=  \langle u d(\xi) \Omega, ux_0\Omega \rangle 	\\
   &=  \langle d(\xi) u\Omega, x_0 u\Omega  \rangle, ~~~(\text{as } u d(\xi)= d(\xi) u) 	
  \end{align*}	
 where the inner product makes sense  in $L^2( M_T^\omega, \varphi^\omega)$, i.e, 
 $$ \langle d(\xi) u\Omega, x_0 u\Omega  \rangle = \lim_{ n \rightarrow \omega } \langle d(\xi) u_n\Omega, x_0 u_n\Omega  \rangle.$$
 Clearly, we  note that the map $ M_T \ni a \rightarrow \langle d(\xi)  u\Omega, ~a  u\Omega  \rangle$ 
is continuous with respect to weak operator topology. Hence, it is enough to show that $$\langle d(\xi)  u\Omega, ~ s(\zeta )  u\Omega  \rangle = 0,$$ for 
 $ \zeta =  \zeta_1 \otimes \zeta_2 \otimes \cdots \otimes \zeta_m $ where $ \zeta_k \in \h_{\R}$. 
 Now suppose $ r_{s} = r(\xi_{s})   $,   $r^*_{s} = r(\J \xi_{s})$ for $ 1 \leq s \leq n $, $l_t = l(\zeta_t) $ and  $l_t^* = l^*(\J\zeta_t)$ for $ 1 \leq t \leq m $.
Then by  using the Wick formula (see Lemma \ref{Wick pdt formula}), it is enough to prove the following  
 $$\langle r_{i_1} \cdots r_{i_l} r^*_{i_{l+1}} \cdots r^*_{i_{n}} u\Omega, 
   ~ ~~\quad l_{j_1} \cdots l_{j_t} l^*_{j_{t+1}} \cdots l^*_{j_{m}} u\Omega  \rangle = 0.$$ 
But then it follows from  \cite[ see Eq. 3.2]{SS}. Then,  by Lemma \ref{inva}, it follows that 
$$ \langle \xi, x_0\Omega \rangle = 0, \text{ for all } \xi \in \F_T( \D)^\perp.$$ 
Therefore, we have  $x_0\Omega \in \F_T( D)$. Consequently, $x \in  M_T( \D_\R)$.

\end{proof}

\begin{thm}\label{apbi}
	Let $ \h_\R =  \h_\R^{ap} + \h_\R^{wm} $ such that $\dim(\h_\R^{ap}) \geq 3 $. Then $M_T( \h_\R)$ has trivial bicentralizer.
\end{thm}

\begin{proof}
First we  note that  $ B_\varphi(M_T) = 	({(M_T^\omega)^{\varphi^\omega}})' \cap M_T$ (see \cite[Proposition.3.3]{HI17}). 
Suppose $ \D_{\R} \subseteq \h_\R^{ap} $ be an $(U_t)$-invariant subspace. Then we employ the Lemma \ref{unitary}, to find a sequence of unitary 
 $ u = (u_n) \subseteq M_T( \D_{\R}) \cap M_T^{ \varphi}$
such that $$ u_n \xrightarrow{ n \rightarrow \infty} 0 \text{ in w.o.t}.$$
Then observe that $ u \in {(M_T^\omega)^{\varphi^\omega}} $. Now suppose  $ x \in ({(M_T^\omega)^{\varphi^\omega}})' \cap M_T$, then we note that $  ux= xu $.
Therefore,  by Lemma \ref{commu}, it follows that $ x \in M_T( \D_{\R})$ and consequently, we obtain
 $$ B_\phi(M_T) \subseteq  M_T(\D_\R,U). $$  
As $\D_\R$ is  any $(U_t)$-invariant subspace of $\h_\R^{ap}$  and $\dim(\h_\R^{ap}) \geq 3 $, so, there exists   two  $(U_t)$-invariant subspaces $\D_\R^1 $ and $\D_\R^2 $  such that  $\D_\R^1 \perp \D_\R^2 $ (orthogonal). Then it follows that 
$$ B_\phi(M_T) \subseteq M_T(\D^1_\R, U) \cap M_T(\D^2_\R, U) = \C1.$$

\end{proof}

\section{ Bicentralizer of $M_T$ when $\h_{\R}^{wm} \neq 0$}

In this section, we show the $M_T$ has trivial bicentralizer when $\h_\R^{wm}  \neq \{0\}$ and without any restriction on the dimension  of  $\h_\R^{ap}$. Finally conclude  that  $M_T$ has trivial bicentralizer whenever it is type   $\mathrm{III}_1$  factor.  We begin with recalling the following definition  from   \cite[3.1]{HI20}. 
\begin{defn}
	Let $\omega\in\beta\N\setminus\N $, $\h$ be a real or a complex Hilbert space and $U :\R\rightarrow \mathcal{U}(\h )$ be  any strongly continuous orthogonal or unitary representation. Then we say  a bounded sequence $(\xi_n)_n\in\ell^{\infty}(\N, H)$ is $(U,\omega)$-equicontinuous if for any $\epsilon > 0$, there exists $\delta > 0$, such that
	\begin{align*}
	\set{n\in\N:\ \sup_{|t|\leqslant\delta}\norm{U_t\xi_n-\xi_n}_{H}<\epsilon} \in\omega.\\
	\end{align*}
\end{defn}

Write   $\mathfrak{E}(\h,U,\omega)$ for the subspace of  all $(U,\omega)$-equicontinuous bounded sequences of  $\ell^{\infty}(\N,\h)$. We note that  $\mathfrak{E}(\h,U,\omega)$  is a closed subspace of $\ell^{\infty}(\N, H)$.
Let  $\h_{\omega}$ be   the ultraproducts Hilbert space of $\h$ (see Section \ref{ultra}). Recall  that $$\h_{\omega}=\ell^{\infty}(\N,\h)/\I_\omega,$$  where 
$$\I_\omega=\set{(\xi_n)_n\in\ell^{\infty}(\N,\h)\ |\ \lim_{n\rightarrow\omega}\norm{\xi_n}_U=0}.$$
Observe that   $ \I_\omega \subseteq \mathfrak{E}(\h,U,\omega)$.

Now consider the unitary map  $(U_{\omega})_t$ on $\h_{\omega}$,  defined by $(U_{\omega})_t(\xi_n)_{\omega}= (U_t\xi_n)_{\omega}$ for $(\xi_n)_{\omega}\in\h_{\omega}$ and for all $t\in\R$. However, note that the map $\R\ni t\mapsto (U_{\omega})_t$ is not necessarily strongly continuous on $\h_{\omega}$. But as   $ \I_\omega \subseteq \mathfrak{E}(\h,U,\omega)$, so,    
$(U_{\omega})_t$ keeps the subspace  $\h_{U,\omega}:=\mathfrak{E}(\h,U,\omega)/\I_\omega \subseteq \h_{\omega}$ invariant. Then note that the map $\R\ni t\mapsto (U_{\omega})_t$ is strongly continuous on $\h_{U,\omega}$.
\begin{lem}
Let  $\h_1$,  $\h_2$ be  two Hilbert spaces and suppose  $\h=\h_1\oplus\h_2$. Then 	
$$ \h_{U,\omega}\simeq(\h_1)_{U,\omega}\oplus(\h_2)_{U,\omega} $$
\end{lem}

\begin{proof}

Clearly $\ell^{\infty}(\N,\h)=\ell^{\infty}(\N,\h_1)\oplus\ell^{\infty}(\N,\h_2)$. Since 
\begin{align*}
\I_\omega=\set{(\xi_n)_n\in\ell^{\infty}(\N,\h)\ |\ \lim_{n\rightarrow\omega}\norm{\xi_n}_U=0},
\end{align*}
and each $\xi_n=\xi_n^1+\xi_n^2$ for $\xi_n^1\in\h_1$, $\xi_n^2\in\h_2$, we have $\lim_{n\rightarrow\omega}\norm{\xi_n}_U=0 \implies\lim_{n\rightarrow\omega}\norm{\xi_n^1}_U+\lim_{n\rightarrow\omega}\norm{\xi_n^2}_U=0$. Hence, $\lim_{n\rightarrow\omega}\norm{\xi_n^1}_U=0$ and $\lim_{n\rightarrow\omega}\norm{\xi_n^2}_U=0$. We denote
\begin{align*}
\I^1_{\omega}=\set{(\xi_n^1)_n\in\ell^{\infty}(\N,\h_1)\ |\ \lim_{n\rightarrow\omega}\norm{\xi_n^1}_U=0}\\ \text{and}\  \I^2_{\omega}=\set{(\xi_n^2)_n\in\ell^{\infty}(\N,\h_2)\ |\ \lim_{n\rightarrow\omega}\norm{\xi_n^2}_U=0}.
\end{align*}
This implies $\I_\omega=\I^1_{\omega}\oplus\I^2_{\omega}$. Therefore, 
we have  $\h_{U,\omega}\simeq(\h_1)_{U,\omega}\oplus(\h_2)_{U,\omega} $.

\end{proof}

Now we  note that $\h_{U,\omega}$ be a large space, indeed for any sequence $ (n_k)$ and $ \xi \in \h_\R$, define $ \xi_k = U_{n_k} \xi $, then it easy to check that $ (\xi_n)_\omega \in \h_{U,\omega}$.  Now  we will write the  following two results  from \cite[Lemma 3.5(1) and Proposition 3.3]{HI20} and for the completeness we include their proofs. It  will be used in the sequel. 

\begin{thm}\label{subspace}
	If $(\h_{\R},U_t)$ is weakly mixing, then there is a separable infinite dimensional   closed subspace $L_{\R} \subset (\h_{\R})_{U,\omega}\ominus \h_{\R}$ such that the restriction  $(L_{\R},(U_{\omega})_t|_{ L_{\R}})$ is almost periodic.
\end{thm}
\begin{proof}
	Since $U$ is weakly mixing, there is an infinite subset $\set{\lambda_n}_{n\in\N}\subset \sigma (A) \cap (1,\infty)$ such that $\sup_n \lambda_n<\infty.$ It follows from \cite[Lemma 3.5(1)]{HI20} that each $\lambda_n$ is in the point spectrum of $A_{\omega}$. Therefore,  for each $\lambda_n$, there exists  a unit vector $\xi^{(n)} = (\xi^{(n)}_m)_{\omega}\in\h_{U,\omega}$  such that $(U_{\omega})_t\xi^{(n)} =\lambda^{it}\xi^{(n)} $ for every $t\in\R$. It follows that $\xi_m^{(n)} \rightarrow 0$ weakly as $m\rightarrow\omega$ and so $\xi^{(n)}\in\h_{U,\omega}\ominus\h$, for each $n\in\N$.  Now, for each $\lambda_n$,  fix a unit eigenvector $\xi^{(n)}\in\h_{U,\omega}\ominus\h$. Consider the two dimensional subspace $L^n$ of $\h_{U,\omega}\ominus\h$ generated by the eigenvector corresponding to $\lambda_n$ and $\lambda_n^{-1}$. Denote the real part of the subspace $L^n$ by $L^n_{\R}$. Clearly $L^n_{\R}\subset (\h_{\R})_{U,\omega}\ominus \h_{\R} $. Now, consider the following  closed subspace $$L_{\R}=\oplus_{n\in\N}L^n_{\R}\subset (\h_{\R})_{U,\omega}\ominus \h_{\R}.$$ It follows that $U_{\omega}|_{L_{\R}}$ is almost periodic and $ \dim( L_\R) = \infty$.

\end{proof}

\begin{defn}
	Let   $(x_n)_n\in\ell^{\infty}(\N, M)$, it is said to be $(\sigma^{\varphi}, \omega)$-equicontinuous  if for any $\epsilon > 0$, there exists $\delta > 0$  such that $\set{n\in\N:\ \sup_{\abs{t}\leqslant\delta}\norm{\sigma_t^{\varphi}(x_n)-x_n}^{\sharp}_{\varphi}<\epsilon}\in\omega$, where $\norm{x}^{\sharp}_{\varphi}= \varphi(xx^*+x^*x)^{1/2}$. The norm $\norm{\cdot}_{\varphi}^{\sharp}$ induces strong* topology on uniformly bounded sets.
\end{defn}

\begin{prop}\label{equi}
	Let $(\xi_n)_n \in\ell^{\infty}(\N, \h_\R)$. The following conditions are equivalent.\\
	(1) $(\xi_n)_n$ is $(U,\omega)$-equicontinuous.\\
	(2) $(s(\xi_n))_n$ is $(\sigma^{\varphi},\omega)$-equicontinuous.\\
	(3) $(s(\xi_n))_n \in M_T^\omega.$\\
	For any $(U, \omega)$-equicontinuous vector $(\xi_n)_n\in\ell^{\infty}(\N, \h_R)$, $\left(\frac{\sqrt{2}}{\sqrt{1+A^{-1}}}\xi_n\right)_n$ is also $(U,\omega)$-equicontinuous.
\end{prop}
\begin{proof}
	For $t\in\R$ and $n\in\N$, we have
	\begin{align*}
	\norm{U_t\xi-\xi}_U=\norm{s(U_t\xi)-s(\xi)}_{\varphi}=\norm{\sigma_{-t}^{\varphi}(s(\xi))-s(\xi)}_{\varphi}
	\end{align*}
	It is known that on uniformly bounded sets, the norm $\norm{\cdot}_{\varphi}$ induces the strong operator topology. Since $s(\xi)$ is self-adjoint for all $\xi\in\h_\R$, we have the equivalence $(1)\iff(2)$. The equivalence $(2)\iff (3)$ follows from \cite[Theorem 1.5]{MT17}.
	
	Let $(\xi_n)_n\in\ell^{\infty}(\N,\h)$ be a $(U,\omega)$-equicontinuous. Notice that 
	\begin{align*}
	\norm{U_t\frac{\sqrt{2}}{\sqrt{1+A^{-1}}}\xi_n-\frac{\sqrt{2}}{\sqrt{1+A^{-1}}}\xi_n}_U&=\norm{\frac{\sqrt{2}}{\sqrt{1+A^{-1}}}U_t\xi_n-\frac{\sqrt{2}}{\sqrt{1+A^{-1}}}\xi_n}_U\\ 
	&\leq \norm{\frac{\sqrt{2}}{\sqrt{1+A^{-1}}}}\norm{U_t\xi_n-\xi_n}_U.
	\end{align*} 
	This implies $\left(\frac{\sqrt{2}}{\sqrt{1+A^{-1}}}\xi_n\right)_n$ is $(U,\omega)$-equicontinuous.
	
\end{proof}

Consider the mixed $q$-Araki-Woods von Neumann algebra $\Gamma_T((\h_{\R})_{U,\omega}, U_{\omega})^{\prime\prime}$  for the same $T$ that defines the mixed $q$-Araki-Woods algebra $M_T$. Observe that $\Gamma_T((\h_{\R})_{U,\omega}, U_{\omega})^{\prime\prime}$ has a faithful unital representation on $\F_T((\h_{\R})_{U,\omega})$ with $\Omega $ as a cyclic and separating vector. We also note that $\Gamma_T((\h_{\R})_{U,\omega}, U_{\omega})^{\prime\prime}$ is a factor (see \cite{BMRW}). The vacuum state on $\Gamma_T((\h_{\R})_{U,\omega}, U_{\omega})^{\prime\prime}$ will be given by $\phi_{\omega}(\cdot)=\inner{ \Omega}{(\cdot)\Omega}_{U,\omega}$.

Denote the $*$-subalgebra generated by Wick products in $\Gamma_T((\h_{\R})_{U,\omega}, U_{\omega})^{\prime\prime}$  by $\widetilde{\Gamma}_T((\h_{\R})_{U,\omega}, U_{\omega})$.

\begin{thm}\label{embedding}
	The map $\iota:\widetilde{\Gamma}_T((\h_{\R})_{U,\omega},U_{\omega})\rightarrow M_T^{\omega}$ defined by $\iota(s((\xi_n)_{\omega}))=(s(\xi_n))_{\omega}$,  extends to a state preserving unital $*$-embedding of $\Gamma_T((\h_{\R})_{U,\omega},U_{\omega})$ into $M_T^\omega$.
\end{thm}
\begin{proof}
	First we note that the vacuum state 	$\phi_{\omega} $  is a faithful normal state  on the $*$-algebra $\widetilde{\Gamma}_T((\h_{\R})_{U,\omega}, U_{\omega})$. Further, the ultraproduct state $ \phi^\omega$ on $M_T^\omega$ is also a faithful normal state. As $\Gamma_T((\h_{\R})_{U,\omega}, U_{\omega})^{\prime\prime}$ is a factor, so, in view of \cite[Remark 2.2]{Bikram-Izumi-Sunder}, to establish the theorem 
	it is enough to check the following. 
	\begin{enumerate}
		\item[(i)]  $ (s(\xi_n))_{\omega} \in M_T^\omega$  for all  $ (\xi_n)_{\omega} \in (\h_{\R})_{U,\omega}$. 
		\item[(ii)]  The map $\iota$ is state preserving on generating set, i.e. $$\varphi^{\omega} ((s(\xi_n))_{\omega})  =\phi_{\omega}(s((\xi_n)_{\omega})), \text{ for all } (\xi_n)_{\omega} \in (\h_{\R})_{U,\omega}.$$
	\end{enumerate}
	Indeed, we note that $(i)$ follows from Proposition \ref{equi}. \\	
	For $(ii)$, we have 
	\begin{align*}
	\varphi^{\omega} (\iota(s((\xi_n)_{\omega})))=\varphi^{\omega} ((s(\xi_n))_{\omega})
	&=\lim_{n\rightarrow\omega}\varphi(s(\xi_n))\\
	&=\lim_{n\rightarrow\omega}\inner{\Omega}{s(\xi_n)\Omega}_{U}\\
	&=\lim_{n\rightarrow\omega}\inner{\Omega}{\xi_n}_{U}\\
	&=\inner{\Omega}{(\xi_n)_{\omega}}_{U,\omega}\\
	&=\inner{\Omega}{s((\xi_n)_{\omega})\Omega}_{U,\omega}=\phi_{\omega}(s((\xi_n)_{\omega})).
	\end{align*}
	Thus, it completes the proof.

\end{proof}

\begin{rem}\label{weakmixing} We note that if $\h_\R^{wm} \neq 0  $, then  similar argument as   \cite[Theorem 8.1]{BM17} can be used to show that   $M_T$ is always type $\mathrm{III}_1$ factor.

\end{rem}

\begin{lem}\label{dim3}
	Assume  $M_T$ is  type  $\mathrm{III}_1$ factor and $\h_\R^{wm} =0  $, then  $ \dim( \h_\R) > 3 .$
\end{lem}
\begin{proof}

	If possible let     $\dim(\h_\R) =2 \text{ or } 3$ and $M_T$ is type $\mathrm{III}_1$ factor. Then $(U_t) \neq 1$,  otherwise it will be type $\mathrm{II}_1$ factor (see  \cite[Theorem 4.5]{MK1}).  Therefore, $( \h_{\R}, U_t)$ will have the following form (see \cite{Shly}):
	\begin{enumerate}
		\item when $\dim(\h_\R) =2 $, then 
		$$
		\h_\R = \R^2, \quad U_t  = \left( \begin{matrix} \cos( t\log \lambda)& - \sin( t\log\lambda) \\
		\sin( t\log\lambda)& \cos( t\log\lambda)  \end{matrix} \right), \text{ for some  }\lambda > 1 \text{ and }
		$$
		\item when $\dim(\h_\R) =3 $, then 
		$$
		\h_\R = \R \oplus \R^2 \text{ and }  \quad U_t  = {id}_t \oplus \left( \begin{matrix} \cos( t\log \lambda)& - \sin( t\log\lambda) \\
		\sin( t\log\lambda)& \cos( t\log\lambda)  \end{matrix} \right), \text{ for some  }\lambda > 1,
		$$\\
		where $(\text{id}_t) $ is the identity representation.
	\end{enumerate}
	Therefore, Connes $S$-invariant (see \cite{CO-S} for the definition of Connes $S$-invariant) will be  $ S( M_T) = \{  \lambda^n : n \in \Z \} \cup \{ 0\}$   (see  \cite[Theorem 4.5]{MK1}) and consequently, $M_T$ will be type  $\mathrm{III}_\lambda$  factor for $ 0< \lambda < 1$.
	Thus, $ \dim( \h_\R) \neq  2 \text{ or } 3$.  Conversely, we have $ \dim( \h_\R) > 3 $.  
\end{proof}

\begin{thm}
	Let $-1<q_{ij}=q_{ji}<1$ be real numbers such that $\sup_{i,j}\abs{q_{ij}}<1$ and let $(\h_{\R}, U_t)$ be a strongly continuous orthogonal representation such that  the mixed $q$-Araki-Woods algebra $M_T$  is a type $\mathrm{III}_1$ factor, then  it  has  trivial bicentralizer.
\end{thm}
\begin{proof}
Suppose   $\h_{\R}=\h_{\R}^{\text{ap}}\oplus\h_{\R}^{\text{wm}}$ is  the unique decomposition of $\h_{\R}$ into almost periodic and weak mixing part with respect to the orthogonal representation $(U_t)$. 

First assume that $\h_{\R}^{\text{wm}} = 0 $.  
Then  note that if $\h_{\R}^{\text{wm}} = 0 $ and  $M_T$ is type $\mathrm{III}_1$ factor,  by Lemma \ref{dim3}, we have  $ \dim( \h_{\R}^{\text{ap}} ) > 3$. Consequently, by Theorem \ref{apbi}, it follows that $$ B_\varphi(M_T) = \C 1.$$

Now assume  $\h_{\R}^{\text{wm}} \neq  0 $. Then by Remark \ref{weakmixing}, it follows that $M_T$ is type $\mathrm{III}_1$ factor. To show that it has trivial bicentralizer, we  consider the real Hilbert space $$\widetilde{\h}_{\R}:=\h_{\R}^{\text{ap}}\oplus\h_{\R}^{\text{wm}} \oplus ((\h_{\R}^{\text{wm}})_{U,\omega}\ominus\h_{\R}^{\text{wm}}),$$ which is a subspace of $(\h_{\R})_{U,\omega}$. 
Now suppose $$ \widetilde{ M_T } = \Gamma_{T}( \widetilde{ \h}_\R , U_\omega   )''.$$

Using Theorem \ref{embedding} and  Remark \ref{condi}, we have the inclusion of von Neumann algebras $  \widetilde{ M_T }  \subseteq  M_T^{\omega}$ with conditional  expectation. Since $\h_{\R}^{\text{wm}}\neq 0$, it follows from Theorem \ref{subspace} that there exists  a    closed subspace $\widetilde{D}_{\R} \subseteq L_\R \subseteq (\h_{\R}^{\text{wm}})_{U,\omega}\ominus\h_{\R}^{\text{wm}}$ such
	that ${\widetilde{D}_{\R}}$ is $U_{\omega}$-invariant.
Then, by Theorem \ref{embedding},  we have  
$$M_T(\widetilde{D}_{\R}, U_{\omega})\subseteq  \widetilde{ M_T } \subseteq M_T^{\omega}$$ 
with conditional expectation (see Remark \ref{condi}).
Therefore, we have the following 
\begin{align*}
B_{\varphi}(M_T) &= 	{\big((M_T^\omega)^{\varphi^\omega}\big)}' \cap M_T, ~ (\text{see Eq. }\ref{rela})\\
&\subseteq   \big(M_T(\widetilde{\D}_{\R}, U_{\omega})^ {\varphi_\omega} \big)^\prime \cap  \widetilde{M_T}\\
&\subseteq  M_T(\widetilde{\D}_{\R}, U_{\omega}) \text{ by Lemma } \ref{commu}.
\end{align*}
Since by Theorem \ref{subspace},    $  L_\R$ is infinite dimensional and $ U_\omega|_{ L_\R}$ is almost periodic, so,  we can find two  orthogonal subspaces 
$\widetilde{\D^1}_{\R}$ and $\widetilde{\D^2}_{\R}$ of $  L_\R$  such that 
$$ B_{\varphi}(M_T) \subseteq M_T(\widetilde{\D^1}_{\R}, U_{\omega}) \cap M_T(\widetilde{\D^2}_{\R}, U_{\omega}) = \C 1.$$
This completes the proof.	
	
\end{proof}


\bibliographystyle{plain}

\end{document}